\documentclass[11pt,a4paper]{article}
\usepackage{hyperref}
\hypersetup{nesting=true,debug=true,naturalnames=true}
\usepackage{graphicx,upref}

\usepackage{amsmath,amssymb,amsthm,enumerate,mathtools,xfrac}
\usepackage{authblk}

\usepackage{xypic} 
\usepackage{mathrsfs} 
\usepackage{xspace}

\newtheorem{thm}{Theorem}[section]
\newtheorem{lem}[thm]{Lemma}

\newtheorem{prop}[thm]{Proposition}
\newtheorem{cor}[thm]{Corollary}

\theoremstyle{definition}
\newtheorem{defn}[thm]{Definition}

\newtheorem{rem}[thm]{Remark}
\newtheorem*{notn}{Notation}

\newcommand{\defbold}{\textbf}

\newcommand{\inv}{^{-1}}

\newcommand{\tdlc}{t.d.l.c.\@\xspace}

\newcommand{\Homeo}{\mathrm{Homeo}}

\newcommand{\bN}{\mathbb{N}}

\newcommand{\bQ}{\mathbb{Q}}

\newcommand{\bZ}{\mathbb{Z}}

\newcommand{\mc}[1]{\mathcal{#1}}

\begin{document}

\title{Equicontinuity, orbit closures and invariant compact open sets for group actions on zero-dimensional spaces}

\author{Colin D. Reid\thanks{The author was an ARC DECRA fellow.  Research supported in part by ARC Discovery Project DP120100996.}}
\affil{University of Newcastle, School of Mathematical and Physical Sciences, Callaghan, NSW 2308, Australia \\ colin@reidit.net}


\maketitle

\begin{abstract}
Let $X$ be a locally compact zero-dimensional space, let $S$ be an equicontinuous set of homeomorphisms such that $1 \in S = S\inv$, and suppose that $\overline{Gx}$ is compact for each $x \in X$, where $G = \langle S \rangle$.  We show in this setting that a number of conditions are equivalent: (a) $G$ acts minimally on the closure of each orbit; (b) the orbit closure relation is closed; (c) for every compact open subset $U$ of $X$, there is $F \subseteq G$ finite such that $\bigcap_{g \in F}g(U)$ is $G$-invariant.  All of these are equivalent to a notion of recurrence, which is a variation on a concept of Auslander--Glasner--Weiss.  It follows in particular that the action is distal if and only if it is equicontinuous.
\end{abstract}

\setcounter{tocdepth}{1}

\section{Introduction}

Let $X$ be a compact metrizable zero-dimensional space and let $G$ be a finitely generated group of homeomorphisms of $G$.  Define the orbit closure relation $R$ on $X$ by setting $(x,y) \in R$ if and only if $y \in \overline{Gx}$.  Then by a result of J.~Auslander, E.~Glasner and B.~Weiss (\cite[Theorem~1.8]{AGW}), $G$ has minimal orbit closures (that is, $R$ is a symmetric relation) if and only if $R$ is a closed subset of $X \times X$; indeed both conditions are equivalent to a kind of recurrence defined in terms of a word metric on $G$.  It follows (\cite[Corollary~1.9]{AGW}) that the action of $G$ is distal if and only if it is equicontinuous.

We prove similar consequences of minimal orbit closures in a more general setting.  In particular, we generalize from finitely generated group actions to group actions that are \defbold{equicontinuously generated}, meaning that the group $G$ has a symmetric generating set that is equicontinuous with respect to some compatible uniformity on the space.  For this purpose we use a variation of the notion of recurrence given in \cite{AGW}, with modifications arising from the change from finitely generated to equicontinuously generated groups.

\begin{defn}\label{defn:cones}
Let $G$ be a group with a specified generating set $S$, such that $1 \in S$ and $S = S\inv$.  Define the associated word metric $|\cdot|$ with respect to $S$.  Write $S^n$ for the set of products of at most $n$ elements of $S$, and let $S^{\infty} =  \langle S \rangle = G$.  Given $g \in G$, define $K(g) = S^{|g|-1}g$.

Suppose $G$ acts on by homeomorphisms on a locally compact Hausdorff space $X$ and let $x \in X$.  Then $x$ is \defbold{$S$-recurrent} for the action if for every net $(g_i)_{i \in I}$ in $G$ such that $|g_i| \rightarrow \infty$ and every neighbourhood $U$ of $x$, there is an integer $n$, a subnet $(g_{i(j)})_{j \in J}$ and $c_j \in K(g_{i(j)})$, such that $c_jx \in U$ and $|c_j| \le n$.

The point $x$ is \defbold{almost periodic} for $G$ if for any neighbourhood $U$ of $x$, and letting $N = \{h \in G \mid hx \in U\}$, then there is a finite subset $F$ of $G$ such that $FN = G$.

A locally compact space is \defbold{zero-dimensional} if it is Hausdorff and has a base consisting of clopen sets.
\end{defn}

We prove the following for equicontinuously generated group actions on locally compact zero-dimensional spaces, which highlights the relationship between orbit closures and compact open invariant sets.

\begin{thm}\label{thmintro:vinfty}
Let $X$ be a locally compact zero-dimensional space, let $G$ be an equicontinuously generated group of homeomorphisms of $X$, let $V$ be an open subset of $X$ and write $V_{\infty} = \bigcap_{g \in G}g(V)$.  Suppose that $V_{\infty}$ is compact.  Then $V_{\infty}$ is open if and only if $\overline{Gx} \cap V_{\infty} = \emptyset$ for every $x \in X \smallsetminus V_{\infty}$.
\end{thm}

Using this result, we prove a generalization of the Auslander--Glasner--Weiss theorem.  (Some of the implications follow with much weaker hypotheses; these are summarized in Proposition~\ref{prop:AGW_any_gen}.)

\begin{thm}\label{thmintro:AGW_generalized}
Let $X$ be a locally compact zero-dimensional space, let $S \subseteq \Homeo(X)$ with $1 \in S$ and $S = S\inv$, such that $S$ is equicontinuous with respect to some compatible uniformity, and let $G = \langle S \rangle$.  Let $U$ be an open subset of $X$, and write $U^*$ for the union of all compact $G$-invariant subsets of $U$.  Then the following are equivalent:
\begin{enumerate}[(i)]
\item\label{mainthm:1} $U^*$ is open in $X$ and consists of $S$-recurrent points.
\item\label{mainthm:2} $U^*$ is open in $X$ and consists of almost periodic points for $G$.
\item\label{mainthm:3} For every compact open subset $V$ of $U$, the intersection $\bigcap_{g \in G}g(V)$ is open.
\item\label{mainthm:4} Given $x \in U$ and $y \in U^*$ such that $y \in \overline{Gx}$, then $x \in \overline{Gy}$.
\item\label{mainthm:5} There is an open neighbourhood $W$ of $U^*$ such that the $G$-orbit closure relation restricted to $W \times W$ is closed.
\item\label{mainthm:6} $U^*$ is open in $X$ and there is a $G$-equivariant quotient map $\phi: U^* \rightarrow Y$, where $Y$ is a locally compact zero-dimensional space such that $G$ acts trivially on $Y$ and minimally on each fibre of $\phi$.
\end{enumerate}
\end{thm}

As a consequence, we obtain a situation where an equicontinuously generated distal action is actually equicontinuous, generalizing for instance \cite[Theorem~2]{Ellis}.

\begin{cor}\label{corintro:distal_to_equi:compact}
Let $X$ be a locally compact zero-dimensional space equipped with a compatible uniformity, let $S \subseteq \Homeo(X)$ with $1 \in S$ and $S = S\inv$, let $G = \langle S \rangle$, and let $x \in X$.  Suppose that $G$ is distal and $\overline{Gx}$ is compact for every $x \in X$.  Then $G$ is equicontinuous if and only if $S$ is equicontinuous.
\end{cor}

\begin{rem}
The original motivation for this article was to develop tools for understanding the dynamics of groups of automorphisms of totally disconnected, locally compact (\tdlc) groups on coset spaces.  In particular, suppose $G$ is a \tdlc group; $S$ is a set of automorphisms of $G$ such that $1 \in S = S\inv$, $S$ is equipped with a compact topology and $(s,g) \mapsto s(g)$ is a continuous map from $S \times G$ to $G$; and $K$ is a closed subgroup of $G$ such that $s(K) = K$ for all $s \in S$.  Then we have an action by homeomorphisms of $H = \langle S \rangle$ on the locally compact zero-dimensional space $G/K$ such that $S$ is equicontinuous with respect to a natural uniformity, and thus Theorem~\ref{thmintro:AGW_generalized} applies to this situation.  The consequences are explored in a second article (\cite{ReidDistal}).
\end{rem}

\subsection*{Acknowledgement}  I thank Eli Glasner for bringing the article \cite{AGW} to my attention, which inspired the present article.

\section{Proof of Theorem~\ref{thmintro:vinfty}}

\begin{notn}
Given a group $G$, a subset $X$ of $G$ and a set of automorphisms $Y$, we say that $X$ is \defbold{$Y$-invariant} if $y(X) = X$ for all $y \in Y$.  Given subsets $X$ and $Y$ of a group $G$, we define $X\inv := \{x\inv \mid x \in X\}$, $XY := \{xy \mid x \in X, y \in Y\}$, $X^n = \{x_1x_2 \dots x_n \mid x_1,x_2,\dots,x_n \in X\}$ and $X^{\infty} = \bigcup_{n \in \bN}X^n$.  All group actions are left actions unless otherwise stated: given a group $G$ acting on a set $X$, we write $gx$ for the image of an element or subset $x$ of $X$ under $g \in G$; for $x \in X$, $Gx$ for the set $\{gx \mid g \in G\}$; and for $Y \subseteq X$, $GY$ for the set $\{gy \mid g \in G, y \in Y\}$.
\end{notn}

\begin{defn}
Let $X$ be a locally compact Hausdorff space with a compatible uniformity $\mc{U}$, let $x \in X$ and let $S \subseteq \Homeo(X)$.  The set $S$ is \defbold{bounded at $x$} if $\overline{Sx}$ is compact; it is \defbold{equicontinuous at $x$} if for every entourage $E \in \mc{U}$, there is a neighbourhood $U$ of $x$ such that 
\[
\forall s \in S \; \forall y \in U: (sx,sy) \in E.
\]
We say $S$ is \defbold{equicontinuous} if $S$ is equicontinuous at every $x \in X$, and \defbold{pointwise bounded} if it is bounded at every $x \in X$.
\end{defn}

We remark that every locally compact Hausdorff space admits a compatible uniformity; for our applications, the choice of uniformity will not be significant in practice.  In particular, it is easily seen that if $\overline{Sx}$ is compact, then all compatible uniformities are equivalent for determining whether or not $S$ is equicontinuous at $x$.

Equicontinuity is preserved by finite unions, and finite products of pointwise bounded equicontinuous families are pointwise bounded equicontinuous.  To see why equicontinuity on its own is not preserved by products, let $f$ be a homeomorphism of $\bQ_2$ such that for all $x \in \bQ_2 \smallsetminus \bZ_2$ and $y \in \bZ_2$, $f(x+y) = f(x) + y/|x|_2$, and consider the family $S = \{x \mapsto f(x)+2^{-n} \mid n \in \bN\}$ of homeomorphisms of $\bQ_2$.  In this case, $S$ is equicontinuous with respect to the $2$-adic uniformity, but $S^2$ is nowhere equicontinuous.

\begin{lem}\label{lem:equi_product}
Let $X$ be a locally compact Hausdorff space equipped with some uniformity $\mc{U}$, let $S,T \subseteq \Homeo(G)$ and let $x \in X$.
\begin{enumerate}[(i)]
\item If $S$ and $T$ are equicontinuous at $x$, then so is $S \cup T$.
\item Suppose that $\overline{Tx}$ is compact, $T$ is equicontinuous at $x$, and $\overline{Tx}$ consists of equicontinuous points for $S$ on $X$.  Then $ST$ is equicontinuous at $x \in X$.  If in addition, $\overline{Sy}$ is compact for every $y \in \overline{Tx}$, then $\overline{STx}$ is compact.
\end{enumerate}
\end{lem}

\begin{proof}
(i)
Let $E$ be an entourage.  Since $S$ and $T$ are equicontinuous at $x$, there are neighbourhoods $V_1$ and $V_2$ of $x$ such that $(sx,sy) \in E$ for all $s \in S$ and $y \in V_1$, and also for all $s \in T$ and $y \in V_2$.  Thus $(sx,sy) \in E$ for all $s \in S \cup T$ and $y \in V_1 \cap V_2$.

(ii)
Let $E$ be an entourage and let $z \in \overline{Tx}$.  Since $S$ is equicontinuous at $z$, there is an open neighbourhood $V_z$ of $z$ with compact closure such that $(sy,sw) \in E$ for all $s \in S$ and $y,w \in V_z$.  In turn there is a smaller open neighbourhood $W_z$ of $z$ and an entourage $E_z$ such that
\[
\forall y \in X \; \forall w \in W_z: \; (y,w) \in E_z \Rightarrow y \in V_z.
\]
Since $\overline{Tx}$ is compact, we can choose finitely many points $z_1,\dots,z_n$ in $\overline{Tx}$ such that $\{W_{z_1},\dots,W_{z_n}\}$ covers $\overline{Tx}$, and obtain an entourage $E' = \bigcap^n_{i=1}E_{z_i}$.  Since $T$ is equicontinuous at $x$, there is a neighbourhood $W$ of $x$ such that $(tx,ty) \in E'$ for all $y \in W$.  Given $t \in T$ and $y \in W$, then $ty \in W_{z_i}$ for some $i$, so $tx \in V_{z_i}$, and hence $(stx,sty) \in E$ for all $s \in S$.  Thus $ST$ is equicontinuous at $x$.

Now suppose in addition that $\overline{Sy}$ is compact for every $y \in \overline{Tx}$.  For each $y \in \overline{Tx}$ choose an entourage $E_y$ such that the set $K_y := \{(z,w) \in E_y \mid z \in \overline{Sy}, w \in X\}$ has compact closure, and let $U_y$ be a neighbourhood of $y$ such that for all $s \in S$, $z \in U_y \Rightarrow (sy,sz) \in E_y$.  Since $\overline{Tx}$ is compact, we have $\overline{Tx} \subseteq \bigcup_{y \in F}U_y$ for a finite subset $F$ of $\overline{Sy}$.  It then follows that $STx \subseteq  \bigcup_{y \in F}K_y$, so $STx$ has compact closure.
\end{proof}

We note the following application of the Arzel\`{a}--Ascoli theorem.

\begin{lem}\label{lem:equi_clopen}
Let $X$ be a locally compact Hausdorff topological space, let $S$ be an equicontinuous set of homeomorphisms of $X$ with respect to some compatible uniformity, and let $U$ be a compact open subset of $X$.
\begin{enumerate}[(i)]
\item The set $V:= \bigcap_{s \in S}s\inv U$ is compact open.
\item If both $\overline{S\inv U}$ and $\overline{SS\inv U}$ are compact, then $\{s\inv U \mid s \in S\}$ is finite.
\end{enumerate}
\end{lem}

\begin{proof}
(i)
We first show there is an entourage $E$ such that
$$x \in U, (x,y) \in E \Rightarrow y \in U.$$
For each $x \in U$, choose $E_x$ to be an entourage such that $(x,y) \in E_x \Rightarrow y \in U$ (this exists since $U$ is open), let $E'_x$ be an entourage such that $(a,b),(b,c) \in E'_x$ implies $(a,c) \in E_x$, and let $U_x$ be the right $E'_x$-neighbourhood of $x$.  Since $U$ is compact, it has a finite cover $\{U_{x_1},\dots,U_{x_n}\}$, so $x \in U$ implies $x \in U_{x_i}$ for some $i$, and hence $(x_i,x) \in E'_{x_i}$.  Now let $E$ be an entourage contained in $\bigcap^n_{i=1}E'_{x_i}$.  Given $(x,y) \in E$ such that $x \in U$, then $(x_i,x) \in E'_{x_i}$ for some $i$ and also $(x,y) \in E'_{x_i}$ for all $i$, so $(x_i,y) \in E_{x_i}$ and hence $y \in U$.

Now let $x \in V$.  Since $S$ is equicontinuous, there is a neighbourhood $W$ of $x$ such that $(sx,sy) \in E$ for all $s \in S$ and $y \in W$.  Since $x \in V$ we have $sx \in U$ for all $s \in S$; the construction of $E$ then ensures $sy \in U$.  We conclude that $V$ contains the neighbourhood $W$ of $x$.  Since $x \in V$ was arbitrary, we have proved that $V$ is open.

To see that $V$ is compact, note that $s\inv U$ is compact for every $s \in S$, since $s\inv$ is a homeomorphism; thus $V$ is an intersection of compact sets.

(ii)
Let $Y = \overline{S\inv U}$ and let $C(Y,X)$ be the space of continuous functions from $Y$ to $X$ with the topology of uniform convergence.  By restricting the domain, let us regard $S$ as a collection of functions from $Y$ to $X$.  Then this family of functions is equicontinuous and uniformly bounded.  Thus by the Arzel\`{a}--Ascoli theorem (see \cite[Theorem~43.15]{Willard}), $S$ has compact closure in $C(Y,X)$.  Since $U$ is clopen in $X$, we then have a continuous map from $C(Y,X)$ to $C(Y,\{0,1\})$, given by $f \mapsto 1_{U} \circ f$.  Note moreover that $C(Y,\{0,1\})$ carries the discrete topology, as any uniformly convergent sequence of functions to $\{0,1\}$ is eventually constant.  Thus the set $\{1_{U} \circ s \mid s \in S\}$ is finite.  The function $1_{U} \circ s$ determines the subset $s\inv U$ of $Y$, so the set $\{s\inv U \mid s \in S\}$ is also finite.
\end{proof}

In particular, we will use a special case Lemma~\ref{lem:equi_clopen} in the proof of Theorem~\ref{thmintro:vinfty}.

\begin{cor}\label{cor:equi_clopen}
Let $X$ be a locally compact zero-dimensional space, let $S \subseteq \Homeo(X)$ with $1 \in S$ and $S = S\inv$, such that $S$ is equicontinuous with respect to some compatible uniformity, and let $G = \langle S \rangle$.  Let $V$ be a compact open subset of $X$ and let $V_0 = \bigcap_{s \in S}s\inv V $.  Define
\[
\text{For } n \in \bN \cup \{\infty\}: V_n = \bigcap_{g \in S^n}gV_0; \text{ for } n \in \bN: W_n = V_n \smallsetminus V_{n+1}.
\]
Then for all $n \in \bN$, $V_n$ is compact open, and for all $0 \le m \le n$, the set $\{gV_n \mid g \in S^m\}$ is finite.
\end{cor}

\begin{proof}
We see by repeated application of Lemma~\ref{lem:equi_clopen}(i) that $V_n$ is open; $V_n$ is compact since it is an intersection of compact sets.

We now claim that $\{gV_n \mid g \in S^m\}$ is finite for all $0 \le m \le n$, using induction on $m$; the case $m=0$ is clear.  Fix $g \in S^m$ for some $0 \le m < n$.  The set $S^2gV_n$ is contained in the compact set $V$, so $\overline{S^2 gV_n}$ is compact; similarly, $\overline{S gV_n}$ is compact.  We now apply Lemma~\ref{lem:equi_clopen}(ii) to $S$ and $g(V_n)$ (recall $S = S^{-1}$ in the present instance), to conclude that $\{s gV_n \mid s \in S\}$ is finite.  By the inductive hypothesis, $\{ gV_n \mid g \in S^m\}$ is finite, so we have proved $\{hV_n \mid h \in S^{m+1}\}$ is finite, completing the induction.
\end{proof}

The next lemma will finish the proof of Theorem~\ref{thmintro:vinfty}; we state the lemma separately in such a way that it can be proved from first principles.  Arguments similar to Lemma~\ref{lem:shell}(i) have appeared in a number of previous articles on totally disconnected, locally compact (\tdlc) groups (see for example the proofs of \cite[Proposition~2.5]{CM} and \cite[Lemma~6.17]{CRW}).

\begin{lem}\label{lem:shell}
Let $X$ be a locally compact zero-dimensional space.  Let $S \subseteq \Homeo(X)$ with $1 \in S$ and $S = S\inv$ and let $V_0$ be a compact open subset of $X$.  Define 
\[
\text{For } n \in \bN \cup \{\infty\}: V_n = \bigcap_{g \in S^n}gV_0; \text{ for } n \in \bN: W_n = V_n \smallsetminus V_{n+1}.
\]
Then the following holds:
\begin{enumerate}[(i)]
\item Given $n \ge 0$, every $G$-orbit intersecting $V_n \smallsetminus V_{\infty}$ also intersects $W_m$ for all $0 \le m \le n$.  In particular, if $W_n$ is nonempty, then so is $W_m$ for all $0 \le m \le n$.
\item If $P_n := (\bigcup_{g \in S^n}gV_n) \smallsetminus V_1$ is compact for all sufficiently large $n \in \bN$, then there is a $G$-orbit $Gx$ that intersects $W_m$ for all $m$ such that $W_m$ is nonempty.
\item Suppose $P_n$ is compact and $V_n$ is open for all sufficiently large $n \in \bN$.  Then $V_{\infty}$ is open if and only if $\overline{Gx} \cap V_{\infty} = \emptyset$ for all $x \in X \smallsetminus V_{\infty}$.
\end{enumerate}
\end{lem}

\begin{proof}
(i)
Let $x \in V_n \smallsetminus V_{\infty}$.  Then $x \in W_{n'}$ for some $n' \ge n$.  Since $x \not\in V_{n'+1}$ there exists $g \in S$ such that $gx \not\in V_{n'}$, and since $x \in V_{n'}$ we have $gx \in V_{n'-1}$.  Thus $gx \in W_{n'-1}$.

By iterating this argument, for all $0 \le m \le n'$ we obtain $h_m \in G$ such that $h_mx \in W_m$.

(ii)
By hypothesis, there exists $b$ such that $P_n$ is compact for all $n \ge b$.  Let $I$ be the set of $n \ge b$ such that $W_n \neq \emptyset$.  Any $G$-orbit $Gx$ intersecting $\bigcap_{n \in I}P_n$ will clearly intersect $V_n \smallsetminus V_{\infty}$ for all $n \in I$; by part (i), $Gx$ then intersects every nonempty $W_n$.  Thus it is enough to show $\bigcap_{n \in I}P_n \neq \emptyset$.

Suppose $x \in P_n$ for $n \ge b$.  Then $\exists g \in S, h \in S^{n-1}: ghx \in V_n$, so $hx \in V_{n-1}$ and hence $x \in P_{n-1}$.  Thus $(P_n)_{n \in I}$ is a descending sequence.

Now suppose $\bigcap_{n \in I}P_n = \emptyset$.  Then by compactness $P_n = \emptyset$ for some $n \in I$, that is, $gV_n \subseteq V_1$ for all $g \in S^n$.  But then $V_n \subseteq \bigcap_{g \in S^n}gV_1 = V_{n+1}$, so $W_n = \emptyset$, contradicting the choice of $n$.  This contradiction completes the proof of (ii).

(iii)
We see that $V_{\infty}$ is a $G$-invariant set; thus if $V_{\infty}$ is open, then certainly $\overline{Gx} \cap V_{\infty} = \emptyset$ for all $x \in X \smallsetminus V_{\infty}$.  On the other hand if $V_{\infty}$ is not open, then it is properly contained in infinitely many of the open sets $V_n$; consequently, infinitely many of the sets $W_n$ are nonempty, and by part (i) it follows that $W_n$ is nonempty for all $n \in \bN$.  Letting $Gx$ be as in part (ii), we see that $x \not\in V_{\infty}$, and there are $g_n \in G$ such that $g_nx \in W_n$ for all $n \in \bN$.  The sequence $(g_nx)_{n \in \bN}$ is confined to the compact set $V_0$, so it has some accumulation point $y$.  Since each of the sets $V_n$ is closed, we see that $y \in \bigcap_{n \in \bN}V_n = V_{\infty}$; thus the intersection $\overline{Gx} \cap V_{\infty}$ is nonempty.
\end{proof}

\begin{proof}[Proof of Theorem~\ref{thmintro:vinfty}]
Take a generating set $S$ for $G$, such that $1 \in S$, $S = S\inv$ and $S$ is equicontinuous with respect to some compatible uniformity.  

Since $V_{\infty}$ is compact and $V$ is locally compact zero-dimensional, there is a compact open subset $W$ of $V$ such that $V_{\infty} \subseteq W$.  We then have $V_{\infty} = \bigcap_{g \in G}gW$.  Now let $V_0 = \bigcap_{s \in S}sW$.  By Corollary~\ref{cor:equi_clopen}, the sets $P_n$, $V_n$ and $V_{\infty}$ as defined in Lemma~\ref{lem:shell} satisfy the hypotheses of Lemma~\ref{lem:shell}(iii): specifically, $V_n$ is open for all $n$, and $P_n = (\bigcup_{g \in S^n}gV_n) \smallsetminus V_1$ is compact since it is a finite union of compact sets minus an open set.  Thus $V_{\infty}$ is open if and only if $\overline{Gx} \cap V_{\infty}$ is empty for every $x \in X \smallsetminus V_{\infty}$, as claimed.
\end{proof}

\section{Recurrence for equicontinuously generated actions}\label{sec:recurrence}

\begin{defn}
Let $G$ be a group, let $X$ be a Hausdorff topological space and let $G$ act on $X$ via the homomorphism $\theta: G \rightarrow \Homeo(X)$.  A \defbold{proximal pair} for the action is a pair $(x,y) \in X \times X$ such that, for some $z \in X$, we have $(z,z) \in \overline{\{(\theta(g)x,\theta(g)y) \mid g \in G\}}$.  An element $x \in X$ is a \defbold{distal point} if, whenever $(x,y)$ is a proximal pair, then $y = x$.  The action is \defbold{distal} if every point is a distal point.

We equip the space $X^X$ of functions from $X$ to $X$ with the topology of pointwise convergence.  The \defbold{enveloping semigroup} $T$ of the action $\theta$ of $G$ on $X$ is the closure of $\theta(G)$ in $X$.
\end{defn}

It is straightforward to verify that $T^2 \subseteq \overline{GT} \subseteq T$, so the enveloping semigroup $T$ is indeed a semigroup under composition; if in addition, $\overline{Gx}$ is compact for every $x \in X$, then $T$ is compact by Tychonoff's theorem.  However in general, $T$ is not a group.

\begin{thm}[Ellis \cite{Ellis}]\label{thm:ellis}
Let $G$ be a group of homeomorphisms of a Hausdorff space $X$ such that $\overline{Gx}$ is compact for every $x \in X$.  Then the following are equivalent:
\begin{enumerate}[(i)]
\item The closure of $G$ in $X^X$ is a group.
\item For every cardinal $\kappa > 0$, $G$ is pointwise almost periodic on $X^\kappa$.
\item There exists a cardinal $\kappa \ge 2$ such that $G$ is pointwise almost periodic on $X^\kappa$.
\item $G$ is distal.
\end{enumerate}
\end{thm}

Distality also implies a certain orbit closure property, and almost periodic points in a locally compact Hausdorff space can be characterized in terms of orbit closures.

\begin{defn}
Let $G$ be a group acting on a topological space $X$.  The \defbold{orbit closure relation} $R$ for the action is given by $(x,y) \in R$ if $y \in \overline{Gx}$.

An action of a group $G$ on a topological space $X$ is \defbold{minimal} if every orbit is dense.  More generally, the action has \defbold{minimal orbit closures} if $G$ acts minimally on $\overline{Gx}$ for every $x \in X$; equivalently, the orbit closure relation $R$ is symmetric.
\end{defn}

\begin{lem}\label{lem:distal_orbit_closure}
Let $H$ be a group acting by homeomorphisms on a Hausdorff topological space $X$.  Let $x,y \in X$ be such that $\overline{Hx}$ is compact and consists of distal points for the action.  Then $H$ acts minimally on $\overline{Hx}$.  Indeed, we have $\overline{Hy} = \overline{Hx}$ whenever $y \in X$ is such that $\overline{Hx} \cap \overline{Hy} \neq \emptyset$.
\end{lem}

\begin{proof}
Let $C = \overline{Hx}$ and let $T$ be the enveloping semigroup of the action of $H$ on $C$.  By Theorem~\ref{thm:ellis}, $T$ is a group; it is also compact, since $C^C$ is compact.  Let $(h_i)$ be a net in $H$ such that $(h_i(y))$ converges to $c \in C$.  By passing to a subnet we may assume that the image of $(h_i)$ in $T$ converges to some limit $t \in T$.  Letting $c' = t\inv(c)$, then $h_i(c')$ converges to $t(t\inv(c)) = c$.  In particular, $(y,c')$ is a proximal pair for the action of $H$.  Since $C$ consists of distal points, we must have $y \in C$, and thus $\overline{Hy} \subseteq \overline{Hx}$.  In particular, $\overline{Hy}$ is compact and consists of distal points for the action, so the same argument applied to $x$ implies that in fact $\overline{Hy} = \overline{Hx}$.
\end{proof}

\begin{lem}[{See also \textit{e.g.} \cite[p31]{GH}}]\label{lem:almost_periodic}
Let $X$ be a locally compact Hausdorff space, let $G \le \Homeo(X)$ and let $x \in X$.  Then the following are equivalent:
\begin{enumerate}[(i)]
\item $\overline{Gx}$ is compact and $G$ acts minimally on $\overline{Gx}$;
\item $x$ is an almost periodic point for $G$.
\end{enumerate}
\end{lem}

\begin{proof}
Let $Y = \overline{Gx}$.

Suppose that $Y$ is compact and that $G$ acts minimally on $Y$, let $V$ be a neighbourhood of $x$ and let $N = \{h \in G \mid hx \in V\}$.  By minimality, $Y$ is covered by the $G$-translates of $V$; since $Y$ is compact, we have $Y \subseteq \bigcup^n_{i=1}g_iV$ for some $g_1,\dots,g_n \in G$.  Given $g \in G$, then $gx \in g_iV$ for some $1 \le i \le n$, so $g\inv_igx \in V$, that is, $g\inv_ig \in N$.  Thus $G = FN$ where $F = \{g_1,\dots,g_n\}$, showing that $x$ is an almost periodic point.

Conversely, suppose that $x$ is an almost periodic point.  Let $Z$ be an nonempty open $G$-invariant subset of $Y$; since $Gx$ is dense in $Y$, we have $Gx \cap Z \neq \emptyset$ and hence $x \in Z$.  Let $V$ be a compact neighbourhood of $x$ such that $V \cap Y \subseteq Z$.   Then $G = FN$ where $F$ is a finite set and $N = \{h \in G \mid hx \in V\}$.  Hence
\[
\forall g \in G \; \exists f \in F: \; gx \in fV \cap Y = f(V \cap Y);
\]
or in other words, $Gx \subseteq F(V \cap Y)$.  Since $F$ is finite, the set $F(V \cap Y)$ is compact and in particular, closed in $Y$.  Since $Gx$ is dense in $Y$, we conclude that $Y = F(V \cap Y) = Z$.  Thus $Y$ is compact and there is no proper nonempty open $G$-invariant subset of $Y$, that is, $G$ acts minimally on $Y$.
\end{proof}

Here is our substitute for \cite[Proposition~1.5]{AGW}.

\begin{lem}\label{lem:replete}
Let $X$ be a locally compact Hausdorff space, let $S \subseteq \Homeo(X)$ with $1 \in S$ and $S = S\inv$ and let $G = \langle S \rangle$.  Let $x \in X$ and let $F$ be a finite subset of $G$.  Then there is a positive integer $n$ such that for all $g \in G$ with $|g| \ge n$, there is $c \in G$ such that $|c|=n$ and $Fc \subseteq K(g)$.
\end{lem}

\begin{proof}
Since $1 \in S = S\inv$ and $S$ generates $G$, there is some $n$ such that $F \subseteq S^{n-1}$.  Suppose $|g| \ge n$, write $g$ as $g = hc$ where $|c| = n$ and $|g| = |h| + |c|$ and let $f \in F$.  Then 
\[
|fcg\inv| = |fh\inv| \le (n-1) + |h|  < |g|,
\]
so $fcg\inv \in S^{|g|-1}$ and hence $fc \in K(g)$.  Thus $Fc \subseteq K(g)$.
\end{proof}

Let us prove some implications in Theorem~\ref{thmintro:AGW_generalized} that hold in greater generality.

\begin{prop}\label{prop:AGW_any_gen}
Let $X$ be a locally compact Hausdorff space, let $G$ be a group of homeomorphisms of $X$, and let $R \subseteq X \times X$ be the $G$-orbit closure relation.
\begin{enumerate}[(i)]
\item\label{any_gen:1} Suppose that there is a $G$-equivariant quotient map $\phi: X \rightarrow Y$, where $Y$ is a locally compact Hausdorff space such that $G$ acts trivially on $Y$ and minimally on each fibre of $\phi$.  Then $R$ is closed.
\item\label{any_gen:2} If $R \cap (W \times W)$ is closed for some open $W \subseteq X$, then for all $x \in X$ and $y \in W$ such that $y \in \overline{Gx}$, we have $x \in \overline{Gy}$.
\item\label{any_gen:3} Let $U$ be an open subset of $X$, and write $U^*$ for the union of all compact $G$-invariant subsets of $U$.  Suppose that $X$ is zero-dimensional and that for every compact open subset $V$ of $U$, the intersection $\bigcap_{g \in G}gV$ is open.  Then $U^*$ is open in $X$ and consists of almost periodic points for $G$.
\item\label{any_gen:4} Let $S$ be a generating set for $G$ with $1 \in S$ and $S = S\inv$, and let $x \in X$ be almost periodic for $G$.  Then $x$ is $S$-recurrent.
\end{enumerate}
\end{prop}

\begin{proof}
(i)
For each $y \in Y$, we see that $\phi\inv(y)$ is a minimal $G$-set.  As a consequence, given $x,y \in X$, we have $(x,y) \in R$ if and only if $\phi(x) = \phi(y)$.  By continuity of $\phi$, this defines a closed relation on $X$.

(ii)
Let $x \in X$ and $y \in W$ such that $y \in \overline{Gx}$. In particular, since $W$ is a neighbourhood of $y$, $gx \in W$ for some $g \in G$.  We also have $g_ix \rightarrow y$ for some net $(g_i)$ in $G$.  We see that $(g_ix,gx) \in R \cap (W \times W)$ eventually as $i \rightarrow \infty$, so $(y,gx) \in R$, since the restriction of $R$ to $W \times W$ is closed.  Thus $gx \in \overline{Gy}$ and hence $x \in \overline{Gy}$.

(iii)
Let $x \in U^*$.  Then there is a compact open neighbourhood $V$ of $\overline{Gx}$ contained in $U$; by hypothesis, we can take $V$ to be $G$-invariant.  Thus $U^*$ is open.  Suppose that some $y \in U^*$ is not almost periodic for $G$.  Then by Lemma~\ref{lem:almost_periodic} there is $z \in \overline{Gy}$ such that $y \not\in \overline{Gz}$.  There is then a compact open neighbourhood $W$ of $\overline{Gz}$ not containing $y$, and hence by our hypothesis, there is an open $G$-invariant neighbourhood $W_{\infty}$ of $\overline{Gz}$ not containing $y$.  But then $\overline{Gy}$ is disjoint from $W_{\infty}$, so $z \not\in \overline{Gy}$, a contradiction.  Thus every $y \in U^*$ is almost periodic.

(iv)
Let $x$ be an almost periodic point for $G$, let $V$ be a neighbourhood of $x$ and let $N$ be the set of $h \in G$ such that $hx \in V$.  Then there is a finite subset $F$ of $G$ such that $FN = G$, in other words $N$ intersects $F^{-1}g$ for all $g \in G$.  Given a net $(g_i)$ in $G$ such that $|g_i| \rightarrow \infty$, then by Lemma~\ref{lem:replete} there is a subnet $(g_{i(j)})_{j \in J}$ such that for all $j \in J$, there is $d_j$ such that $F\inv d_j \subseteq K(g_{i(j)})$, and moreover $|d_j|$ is bounded. We then choose $f_j \in F\inv$ such that $f_jd_j \in N$ and set $c_j = f_jd_j$; then $|c_j|$ is bounded and $c_jx \in V$ for all $j \in J$.  Hence $x$ is $S$-recurrent.
\end{proof}

We are now ready to prove the main theorem.

\begin{proof}[Proof of Theorem~\ref{thmintro:AGW_generalized}]
Let us first note that for every $x \in U^*$, the set $\overline{Gx}$ is a compact subset of $U^*$.  Indeed, given $x \in U^*$, then $x \in V \subseteq U^*$ for some compact $G$-invariant set $V$, and then $\overline{Gx}$ is a closed subset of $V$.  It follows by Lemma~\ref{lem:equi_product} and induction on $n$ that $S^n$ is equicontinuous at every $x \in U^*$.

We list the implications that have already been proved, together with the place they are proved:

(\ref{mainthm:6}) $\Rightarrow$ (\ref{mainthm:5}): Proposition~\ref{prop:AGW_any_gen}(\ref{any_gen:1}) (applied to the space $U^* = W$);

(\ref{mainthm:5}) $\Rightarrow$ (\ref{mainthm:4}): Proposition~\ref{prop:AGW_any_gen}(\ref{any_gen:2});

(\ref{mainthm:4}) $\Rightarrow$ (\ref{mainthm:3}): Theorem~\ref{thmintro:vinfty};

(\ref{mainthm:3}) $\Rightarrow$ (\ref{mainthm:2}): Proposition~\ref{prop:AGW_any_gen}(\ref{any_gen:3});

(\ref{mainthm:2}) $\Rightarrow$ (\ref{mainthm:1}): Proposition~\ref{prop:AGW_any_gen}(\ref{any_gen:4}).

Now suppose that (\ref{mainthm:1}) is satisfied.  We will first use (\ref{mainthm:1}) to prove (\ref{mainthm:5}), and then, since (\ref{mainthm:5}) has been shown to imply (\ref{mainthm:1})-(\ref{mainthm:4}), we can use conditions (\ref{mainthm:3}) and (\ref{mainthm:4}) to prove (\ref{mainthm:6}) and thereby complete the cycle of implications between the conditions (\ref{mainthm:1})-(\ref{mainthm:6}).

Let $R$ be the $G$-orbit closure relation on $X$.  Since $U^*$ is open, we simply set $W = U^*$.  Suppose $(y_i,z_i)_{i \in I}$ is a net of pairs of points in $R$ such that $(y_i,z_i) \rightarrow (y,z) \in U^* \times U^*$ but $(y,z) \not\in R$, that is, $z \not\in \overline{Gy}$.  There is then a compact open neighbourhood $V$ of $z$ that is contained in $U^*$ and disjoint from $\overline{Gy}$; by passing to a subnet we may assume $y_i \in U^*$ and $z_i \in V$ for all $i \in I$.

Given $i \in I$, let $g_i$ be an element of $G$ of least possible word length such that $g_iy_i \in V$.  Such elements $g_i$ exist since $z_i \in \overline{Gy_i} \cap V$ for all $i$.  Since $V$ is compact, by passing to a subnet we may assume $g_iy_i \rightarrow w$ for some $w \in V$.  Suppose that $|g_i| \not\rightarrow \infty$.  Then we can pass to a subnet on which $g_i$ is bounded, say $|g_i| \le n$.  Since $S^n$ is equicontinuous at $y$ and $y_i \rightarrow y$, for every entourage $E$ there is $i_E \in I$ such that for all $i > i_E$, we have $(g_iy,g_iy_i) \in E$.  It then follows that $g_iy \rightarrow w$, a contradiction, since $w \in V$ whereas $\overline{Gy}$ is disjoint from $V$.  From this contradiction we conclude that in fact  $|g_i| \rightarrow \infty$, and hence $|g\inv_i| \rightarrow \infty$.

Since $w$ is an $S$-recurrent point, on passing to a subnet we may ensure the existence an integer $n$ and elements $c_i$ of $K(g\inv_i)$, such that $c_iw \in V$ and $|c_i| \le n$ for all $i \in I$.  Since $S^n$ is equicontinuous at $U^*$, the elements $\{c_i \mid i \in I\}$ form an equicontinuous family at $w$.  Since $V$ is compact open, there is an entourage $E$ such that
\[
w_1 \in V, (w_1,w_2) \in E \Rightarrow w_2 \in V;
\]
let $W_E$ be a neighbourhood of $w$ such that $(c_iw,c_iw') \in E$ for all $w' \in W_E$ and $i \in I$.  Then for all sufficiently large $i$, we have $g_iy_i \in W_E$, so $(c_iw,c_ig_iy_i) \in E$, and hence $c_ig_iy_i \in V$.  On the other hand, for each $i$, we can write $c_i$ as $c_i = h_ig\inv_i$ where $|h_i| < |g_i|$; thus $c_ig_iy_i = h_iy_i \not\in V$ by the minimality of $|g_i|$.  From this contradiction, we conclude that in fact the restriction of $R$ to $U^* \times U^*$ is closed, so (\ref{mainthm:1}) implies (\ref{mainthm:5}).

Condition (\ref{mainthm:4}) ensures that given $x,y \in U^*$, we have $(x,y) \in R$ if and only if $\overline{Gx} = \overline{Gy}$, so $R$ restricts to a closed equivalence relation $R'$ on $U^*$.  Write $[x]$ for the $R'$-class of $x$, let $Y$ be the set of $R'$-classes and let $\phi: U^* \rightarrow Y$ be given by $x \rightarrow [x]$.  We see that $\phi(gx) = \phi(x)$ for all $x \in U^*$, so $\phi$ is $G$-equivariant with respect to the trivial action on $Y$, and moreover $G$ acts minimally on each fibre.  We equip $Y$ with the quotient topology induced by $\phi$.  The fibres of $\phi$ are of the form $\overline{Gx}$ for $x \in U^*$, so they are compact. Let $\tau$ be the topology of $Y$ and let $\mc{L}$ be the set of compact open $G$-invariant subsets of $U^*$.  Given $V \in \mc{L}$ we see that $V$ is a union of $R'$-classes, so $\phi\inv(\phi(V)) = V$ and in particular $\phi(V), Y \smallsetminus \phi(V) \in \tau$.  On the other hand, given $x \in U^*$ and a neighbourhood $W$ of $x$ containing $\overline{Gx}$, then $W$ contains a compact open neighbourhood of $\overline{Gx}$, and by (\ref{mainthm:3}) we obtain $V \in \mc{L}$ such that $\phi\inv(y) \subseteq V \subseteq W$, so $\{\phi(V) \mid V \in \mc{L}\}$ is a base for $\tau$ and $\tau$ is Hausdorff.  We conclude that $Y$ is locally compact and zero-dimensional, completing the proof of (\ref{mainthm:6}).
\end{proof}

\begin{rem}
Let $X$ be a locally compact zero-dimensional space, let $G \le \Homeo(X)$ be equicontinuously generated, let $Y$ be the interior of the set of almost periodic points of $X$ and let $U$ be a $G$-invariant open subset of $X$.  As a consequence of Theorem~\ref{thmintro:AGW_generalized}, we see that if $U^* \subseteq Y$, then $U^* = U \cap Y$ and $U$ satisfies all of the conditions  (\ref{mainthm:1})--(\ref{mainthm:6}) of the theorem, whereas if $U^* \not\subseteq Y$, then $U$ fails to satisfy any of the conditions  (\ref{mainthm:1})--(\ref{mainthm:6}) of the theorem.  Thus the set $Y$ can be characterized in several equivalent ways, using each of the conditions (\ref{mainthm:1})--(\ref{mainthm:6}) individually.
\end{rem}

As a consequence of Theorem~\ref{thmintro:AGW_generalized}, we obtain a more general version of \cite[Corollary~1.9]{AGW}.  We state a version here where a single orbit is assumed to have compact closure; Corollary~\ref{corintro:distal_to_equi:compact} will then follow immediately as a special case.

\begin{cor}\label{cor:distal_to_equi}
Let $X$ be a locally compact zero-dimensional space equipped with a compatible uniformity, let $S \subseteq \Homeo(X)$ with $1 \in S$ and $S = S^{-1}$, let $G = \langle S \rangle$, and let $x \in X$.  Suppose that $\overline{Gx}$ is compact and that there is a neighbourhood of $\overline{Gx}$ in $X$ consisting of distal points for $G$.  Then the following are equivalent:
\begin{enumerate}[(i)]
\item There is a closed $G$-invariant neighbourhood $U$ of $x$ consisting of equicontinuous points for $S$.
\item There is a compact $G$-invariant neighbourhood $V$ of $x$ on which $G$ acts equicontinuously.
\end{enumerate}
\end{cor}

\begin{proof}
Clearly (ii) implies (i).  Suppose that (i) holds.  Then we are in the situation of Theorem~\ref{thmintro:AGW_generalized}.  Let $U'$ be a neighbourhood of $\overline{Gx}$ in $X$ consisting of distal points for $G$.  Let $Y = U \cap \bigcup_{g \in G}gU'$; then $Y$ is a locally compact $G$-invariant Hausdorff space on which $G$ is distal and $S$ is equicontinuous.  Given $y \in Y$, if the $G$-orbit of $y$ has compact closure, and $y \in \overline{Gz}$ for some $z \in Y$, then $z \in \overline{Gy}$ by Lemma~\ref{lem:distal_orbit_closure} (here we take all closures in $Y$).  Thus the set of homeomorphisms induced by $S$ on $Y$ satisfies the hypotheses of Theorem~\ref{thmintro:AGW_generalized}(\ref{mainthm:4}).  Now take a compact open neighbourhood $W$ of $\overline{Gx}$ in $X$ such that $W \subseteq Y$.  We apply Theorem~\ref{thmintro:AGW_generalized}(\ref{mainthm:3}) to the action of $G$ on $Y$, to obtain a $G$-invariant set $V = \bigcap_{g \in F}g(W)$ where $F$ is finite.  It is clear that $S$ is equicontinuous on $V \times V$.  By Theorem~\ref{thm:ellis}, the action of $G$ on $V \times V$ has minimal orbit closures, so the hypotheses of Theorem~\ref{thmintro:AGW_generalized}(\ref{mainthm:4}) are satisfied for the action of $S$ on $V \times V$.  Applying Theorem~\ref{thmintro:AGW_generalized}(\ref{mainthm:3}) to the latter action, for any compact open subset $E$ of $V \times V$, there is a finite subset $F'$ of $G$ such that $E_{\infty} := \bigcap_{g \in G}g(E) = \bigcap_{g \in F'}g(E)$.  In particular, every neighbourhood $E$ of the diagonal in $V \times V$ contains a compact open neighbourhood $E'$ of the diagonal, which in turn contains a \emph{$G$-invariant} neighbourhood $(E')_{\infty}$ of the diagonal.  Thus $G$ acts equicontinuously on $V$, proving (ii).
\end{proof}


\begin{thebibliography}{99}

\bibitem{AGW}
J. Auslander, E. Glasner and B. Weiss, On recurrence in zero dimensional flows, {\it Forum Math.} {\bf 19} (2007), 107--114. 

\bibitem{CM}
P.-E. Caprace and N. Monod, Decomposing locally compact groups into simple pieces, {\it Math. Proc. Cambridge Philos. Soc.} {\bf 150} (2011), no. 1, 97--128.

\bibitem{CRW}
P.-E. Caprace, C. D. Reid and G. A. Willis, Locally normal subgroups of totally disconnected groups. Part II: Compactly generated simple groups, {\it Forum Math. Sigma} {\bf 5} (2017), e12, 89pp.

\bibitem{Ellis}
R. Ellis, Distal transformation groups, {\it Pacific J. Math.} {\bf 8} (1958), 401--405. 

\bibitem{GH}
W. Gottschalk and G. Hedlund, Topological dynamics, {\it Amer. Math. Soc. Colloquium Publications}, vol. 36, Providence, 1955.

\bibitem{ReidDistal}
C. D. Reid, Distal actions on coset spaces in totally disconnected, locally compact groups, submitted, preprint available at arXiv:1610.06696.

\bibitem{Willard}
S. Willard, General topology. {\it Addison--Wesley, Reading}, 1970.

\end{thebibliography}
\end{document}